\documentclass{amsart}

\usepackage{amsfonts,amssymb,amsmath,amsthm}
\usepackage{tikz}
\usepackage{hyperref}

\newcommand{\kk}{\mathbb{K}}
\newcommand{\N}{\mathbb{N}}
\newcommand{\ZZ}{\mathbb{Z}}
\newcommand{\walpha}{\widehat{\alpha}}

\newcommand{\R}{\mathbb{R}}
\newcommand{\m}{\mathfrak{m}}

\DeclareMathOperator{\Ass}{Ass}
\DeclareMathOperator{\NP}{NP}
\DeclareMathOperator{\SP}{SP}

\DeclareMathOperator{\reg}{\mbox{reg}}
\DeclareMathOperator{\supp}{supp}
\DeclareMathOperator{\link}{link}
\DeclareMathOperator{\STAR}{star}

\newtheorem{thm}{Theorem}[section]

\newtheorem{question}[thm]{Question}
\newtheorem{proposition}[thm]{Proposition}
\newtheorem{corollary}[thm]{Corollary}
\newtheorem{theorem}[thm]{Theorem}
\newtheorem{lemma}[thm]{Lemma}

\theoremstyle{definition}
\newtheorem{definition}[thm]{Definition}

\newtheorem{example}[thm]{Example}
\newtheorem{notation}[thm]{Notation}
\newtheorem{remark}[thm]{Remark}

\title{Asymptotic resurgence via integral closures}

\author[M.~DiPasquale]{Michael DiPasquale}     
\address{Michael DiPasquale\\     
	Department of Mathematics\\     
	Oklahoma State University\\     
	Stillwater\\
	OK \ 74078-1058\\     
	USA}     
\email{Michael.DiPasquale@colostate.edu}
\urladdr{\url{https://midipasq.github.io/}}   
\author[C.A.~Francisco]{Christopher A. Francisco}
\address{Christopher A. Francisco\\
	Department of Mathematics\\     
	Oklahoma State University\\     
	Stillwater\\
	OK \ 74078-1058\\     
	USA}    
\email{chris.francisco@okstate.edu}
\urladdr{\url{https://math.okstate.edu/people/chris/}}
\author[J.~Mermin]{Jeffrey Mermin}
\address{Jeffrey Mermin\\     
	Department of Mathematics\\     
	Oklahoma State University\\     
	Stillwater\\
	OK \ 74078-1058\\     
	USA}
\email{mermin@math.okstate.edu}     
\urladdr{\url{https://math.okstate.edu/people/mermin/}}   
\author[J.~Schweig]{Jay Schweig}
\address{Jay Schweig\\
	Department of Mathematics\\     
	Oklahoma State University\\     
	Stillwater\\
	OK \ 74078-1058\\     
	USA}
\email{jay.schweig@okstate.edu}
\urladdr{\url{https://math.okstate.edu/people/jayjs/}}

\begin{document}

\begin{abstract}
Given an ideal in a polynomial ring, we show that the asymptotic resurgence studied by Guardo, Harbourne, and Van Tuyl can be computed using integral closures.  As a consequence, the asymptotic resurgence of an ideal is the maximum of finitely many ratios involving Waldschmidt-like constants (which we call \textit{skew} Waldschmidt constants) defined in terms of Rees valuations.  We use this to prove that the asymptotic resurgence coincides with the resurgence if the ideal is normal (that is, all its powers are integrally closed).

For a monomial ideal the skew Waldschmidt constants have an interpretation involving the symbolic polyhedron defined by Cooper, Embree, H\`a, and Hoefel.  Using this intuition we provide several examples of squarefree monomial ideals whose resurgence and asymptotic resurgence are different. 
\end{abstract}

\maketitle

\section{Introduction}\label{s:intro}
If $I$ is a homogeneous ideal in the polynomial ring $S= \kk[x_1,\ldots,x_n]$, there are different notions of taking powers of $I$.  The ordinary power $I^r$ is algebraically the most obvious choice, but the geometric information it contains can be quite difficult to understand (there may be many embedded primes).  On the other hand, the symbolic power $I^{(s)}=\cap_{P\in\Ass(I)} (I^sS_P\cap S)$ is geometrically more natural; for instance (if $I$ is radical) the Zariski-Nagata theorem shows that $I^{(s)}$ consists of polynomials vanishing to order $s$ along the projective variety defined by $I$.  However, the algebraic properties of $I^{(s)}$ are much more opaque.  It may be quite difficult even to determine the minimum degree of a polynomial in $I^{(s)}$ (denoted $\alpha(I^{(s)})$).  For example, a famous unresolved conjecture of Nagata~\cite{N61} asserts that if $I\subset\kk[x,y,z]$ is the ideal of $r>9$ very general points in the projective plane, then $\alpha(I^{(s)})>s\sqrt{r}$.

There is a large body of research comparing regular and symbolic powers -- we refer the reader to~\cite{DDSGHNB18} for a recent survey.  Our focus in this paper is on the \textit{containment problem}; this is the study of the set of pairs $(r,s)$ of integers so that $I^{(s)}\subset I^r$.  The containment problem has attracted a great deal of attention since the pioneering work of Swanson~\cite{S00} which led to the seminal papers of Ein, Lazarsfeld, and Smith~\cite{ELS01} and Hochster and Huneke~\cite{HH02} (more on this below).  A few of the many papers written on this problem in the last ten years are~\cite{BDHKKSS09,BH10Resurgence,DST13,GHV13,BCH14,HS15,DHNSST15,CEHH17}.

For the containment problem to be interesting, one should know that for a fixed $r$, $I^{(s)}\subset I^r$ for all $s\gg 0$.  Even more is true; a groundbreaking result of Swanson~\cite{S00} is that $I^{(hr)}\subset I^r$ for some constant $h$ (\emph{a priori} dependent on $I$).  This pioneering work led to the celebrated result, due to Ein, Lazarsfeld, and Smith~\cite{ELS01} and Hochster and Huneke~\cite{HH02}, that $I^{(hr)}\subset I^r$ for an ideal $I$ with big height $h$ in a regular Noetherian ring (such as the polynomial ring $S$).

To capture how large an $h$ is necessary in the containment $I^{(hr)}\subset I^r$, Bocci and Harbourne define the resurgence of $I$ in~\cite{BH10} as
\[
\rho(I)=\sup\left\lbrace\frac{s}{r}:I^{(s)}\not\subset I^r\right\rbrace.
\]
By~\cite{ELS01,HH02}, $\rho(I)\le n-1$.  In~\cite{GHV13}, Guardo, Harbourne, and Van Tuyl introduce the asymptotic resurgence of $I$,
\[
\rho_a(I)=\sup\left\lbrace \frac{s}{r}:I^{(st)}\not\subset I^{rt} \mbox{ for all }t\gg 0\right\rbrace,
\]
and study it for smooth projective schemes.  Although there is an evident inequality $\rho_a(I)\le \rho(I)$, there are several known examples where $\rho_a(I)\neq \rho(I)$ (see~\cite{DHNSST15} and the recent preprint~\cite{BDHHSS17}).

Since $\rho(I)$ and $\rho_a(I)$ are difficult to compute, it has been useful to study lower bounds.  One important lower bound on the (asymptotic) resurgence, introduced in~\cite{BH10}, is defined in terms of the \textit{Waldschmidt constant} $\walpha(I)$, defined as
\[
\walpha(I)=\lim_{s\to\infty}\frac{\alpha(I^{(s)})}{s}.
\]
Bocci and Harbourne show that $1\le\frac{\alpha(I)}{\walpha(I)}\le\rho_a(I)\le\rho(I)$.

Bocci and Harbourne also consider upper bounds on the resurgence.  They show in \cite{BH10} that $\rho(I)\le \frac{\reg(I)}{\walpha(I)}$ whenever $I$ defines a zero-dimensional scheme.  The stronger inequality $\rho_a(I)\le \frac{\omega(I)}{\walpha(I)}$ holds if $I$ defines a smooth scheme~\cite{GHV13} (where $\omega(I)$ is the largest degree of a minimal generator of $I$).  The upper bounds in terms of $\omega(I)$ and $\reg(I)$ can fail if the ideal is not smooth or zero dimensional, as the following example shows.

\begin{example}\label{ex:simpleexample}
Let $I=( ab,ac,bc )$ be the ideal of a three-cycle (geometrically, the ideal of three generic points in $\mathbb{P}^2$).  It is well known that $\rho(I)=\frac{4}{3}$.  Now let $J=( abd,acd,bcd)=dI$.  Geometrically, $J$ is the ideal of three generic lines with the hyperplane at infinity in $\mathbb{P}^3$.  One can check that $\walpha(J)=\frac{5}{2}$ and $\reg(J)=3$, so
\[
\frac{\alpha(J)}{\walpha(J)}=\frac{\omega(J)}{\walpha(J)}=\frac{\reg(J)}{\walpha(J)}=\frac{3}{\frac{5}{2}}=\frac{6}{5}	
\]
However, $J^r=d^rI^r$ and $J^{(s)}=d^sI^{(s)}$, so $\rho(J)=\rho(I)=\frac{4}{3}>\frac{6}{5}$.
\end{example}

Example~\ref{ex:simpleexample} highlights the need for ways to compute and bound the resurgence and asymptotic resurgence of ideals defining schemes that are not smooth or zero-dimensional.  In this paper we provide a new characterization of asymptotic resurgence that is computationally effective at least for monomial ideals.  Our work hinges on the observation that one may replace $I^{rt}$ by its integral closure $\overline{I^{rt}}$ when computing $\rho_a(I)$.  More precisely, we prove in Section~\ref{s:AsymptoticResurgence} that
\[
\rho_a(I)=\sup\left\{\frac{s}{r}: I^{(st)}\not\subset \overline{I^{rt}}\mbox{ for all } t\gg 0 \right\}=\sup\left\{\frac{s}{r}: I^{(s)}\not\subset \overline{I^{r}}\right\}.
\]
It is immediate from these equalities that $\rho_a(I)=\rho(I)$ if $I$ is normal (that is, all powers of $I$ are integrally closed).  This partially answers the question of when $\rho(I)=\rho_a(I)$, raised at the end of~\cite{GHV13}.

Once we have replaced powers of $I$ by their integral closures, we can determine $\rho_a(I)$ using valuations.  Given a valuation $v:S\to\ZZ$, we define a Waldschmidt-like constant (which we call a \textit{skew} Waldschmidt constant)
\[
\widehat{v}(I)=\lim_{s\to\infty}\frac{v(I^{(s)})}{s}.
\]
Our main result is Theorem~\ref{thm:generalizedWBound}, which computes the asymptotic resurgence in terms of the skew Waldschmidt constants:  
\[
\rho_a(I)=\sup\left\{\frac{v(I)}{\widehat{v}(I)}:v(I)>0\right\}.
\]
In fact, we show that one only needs to consider the finitely many Rees valuations of $I$ in the above supremum, so the asymptotic resurgence can be characterized as the maximum of finitely many ratios involving skew Waldschmidt constants.

While Theorem~\ref{thm:generalizedWBound} in theory gives an algorithm for computing asymptotic resurgence, it is usually impractical to compute the Rees valuations of an ideal.  This particular hurdle is much easier for the class of monomial ideals, where the Rees valuations may be identified with the linear functionals corresponding to the bounding (affine) hyperplanes of the Newton polyhedron.
Our key insight is that skew Waldschmidt constants can be obtained as the minimum value of a linear functional on the \emph{symbolic polyhedron} defined by Cooper, Embree, H\`a, and Hoefel in \cite{CEHH17}.  Consequently, computing the asymptotic resurgence of any given monomial ideal can be reduced entirely to linear programming.  In particular, since the Newton polyhedron of a monomial ideal has rational vertices, the asymptotic resurgence of monomial ideals is rational.

The paper is organized, for the sake of exposition, with our main results (summarized above) in the final section (Section~\ref{s:AsymptoticResurgence}) while Sections~\ref{s:squarefree} and~\ref{s:Examples} are devoted to developing the intuition to our approach in the context of squarefree monomial ideals and giving a number of examples.  In Section~\ref{s:squarefree} we show that the asymptotic resurgence can be computed by solving linear programs over the symbolic polyhedron, extending earlier work (\cite{CEHH17,Waldschmidt16}) that took a similar approach to computing the Waldschmidt constant. 

In Section~\ref{s:Examples} we exhibit several squarefree monomial ideals whose asymptotic resurgence is strictly less than their resurgence.  Prior to our work, it was shown in~\cite{DHNSST15} and in the preprint~\cite{BDHHSS17} that asymptotic resurgence can differ from resurgence. These examples are ideals of points for which the symbolic cube is not contained in the regular square, violating a containment conjectured by Harbourne~\cite[Conjecture~8.4.3]{BDHKKSS09}. One way in which our examples are qualitatively different is that 
monomial ideals satisfy Harbourne's containment conjecture \cite[Example~8.4.5]{BDHKKSS09}.

We also prove in Section~\ref{s:Examples} that the asymptotic resurgence of a squarefree monomial ideal generated in degree two can be computed solely in terms of the Waldschmidt constant, thus confirming (asymptotically) a conjecture communicated to us at the 2017 BIRS-CMO workshop in Oaxaca.

\vspace{5 pt}

\noindent\textbf{Acknowledgements.} The content of this paper arose out of various attempts to prove the aforementioned conjecture which we encountered at the 2017 BIRS-CMO workshop in Oaxaca.  We are grateful to BIRS-CMO and the organizers and participants of the Oaxaca workshop and for the inspiring discussions and problems we encountered there.  We are especially indebted to Adam Van Tuyl for introducing this conjecture in Oaxaca and for his insight into this problem.

We also thank Alexandra Seceleanu, T\`ai H\`a, Craig Huneke, and Elo\'isa Grifo for their comments on early drafts of this paper.  We are grateful to Kuei-Nuan Lin for alerting us to the papers~\cite{OH98,SVV98}, where the characterization for normality of edge ideals is completed.

\section{Asymptotic resurgence for squarefree monomial ideals}\label{s:squarefree}

We begin by developing the intuition about asymptotic resurgence for squarefee monomial ideals. Squarefree monomial ideals have been of considerable interest in studying the relationship between symbolic and regular powers, and have the dual advantages (even over monomial ideals) of being intersections of prime (rather than primary) ideals, and of having no embedded primes.  In particular, the following result is both standard and very useful.

\begin{theorem}\label{t:whysquarefree}
Let $I\subset S$ be a squarefree monomial ideal.  Then all associated
primes of $I$ are generated by a subset of the variables, and for all
$s$ we have
\[
I^{(s)}=\bigcap_{P\in \Ass(I)} P^{s}.
\]
\end{theorem}

\begin{remark}
Although we state most of the results in this section for squarefree monomial ideals, the major results hold in more generality. In particular, Theorem~\ref{t:squarefreeresurgence} holds for arbitrary monomial ideals. The chief difference is that the definition of symbolic powers is considerably easier for squarefree ideals, and so one of our key tools, the symbolic polyhedron, requires a more complicated definition.  One can recover everything except Proposition~\ref{p:Symbolicpoly} and Lemma \ref{l:hyperplanes} for monomial ideals without embedded primes, at the price of some intuition and a couple extra pages of technical work.  In view of the more general results in Section \ref{s:AsymptoticResurgence}, we have opted to keep the intuition.  
\end{remark}

\begin{remark}\label{r:IntClosure}
Throughout the section, we assume Corollary~\ref{cor:rhoabar=rhobar}, which asserts that $\rho_a(I)=\sup \left\{\frac{s}{r}: I^{(s)} \not \subset \overline{I^r} \right\}$. We delay the proof of this statement to Section~\ref{s:AsymptoticResurgence} because there is no special intuition in the squarefree case.
\end{remark}

The symbolic powers, and integral closures of regular powers, of
squarefree monomial ideals are best understood in terms of exponent
vectors and related polyhedra.  We require some notation.

\begin{notation}\label{n:exponentvector}
  Let $m\in S$ be any monomial (squarefree or otherwise).  Write
  $m=\prod x_{i}^{e_{i}}$.  We say that the vector $\mathbf{v}=(e_{1},\dots,
  e_{n})$ is the \emph{exponent vector} of $m$, and write $m=\mathbf{x}^{\mathbf{v}}$.
\end{notation}

Observe that $\mathbf{v}$ is always in the (closed) first orthant of
$\mathbb{R}^{n}$.  The geometry of the space of exponent vectors is central to the rest
of the section.  We begin by recalling two important polyhedra.

\begin{definition}\label{d:NewtonPoly}
  Let $I$ be a monomial ideal.  The \emph{Newton polyhedron} of $I$ is
  the convex hull of the exponent vectors occuring in $I$,
  \[
  \text{NP}(I)=\text{conv}\{\mathbf{v}:m=\mathbf{x}^\mathbf{v}\in I\}.
  \]
\end{definition}

\begin{remark}
  The Newton polyhedron is the convex hull of all monomials in $I$,
  not merely the minimal generators.  The convex hull of the (exponent
  vectors of) minimal generators is much smaller, and some texts
  distinguish the two objects by calling the smaller one the
  \emph{Newton polytope} of $I$.  Using this language the Newton
  polyhedron is the Minkowski sum of the Newton polytope and the first
  orthant.  
\end{remark}

The Newton polyhedron characterizes the integral closures of powers of
the ideal.

\begin{proposition}\label{p:Polyhedronclosure}
  Suppose that $\NP$ is the Newton polyhedron of $I$, and fix a monomial
  $m=\mathbf{x}^\mathbf{v}$.  Then $m\in \overline{I^{r}}$ if and only if
  $\frac{\mathbf{v}}{r}\in \NP$.
\end{proposition}

The symbolic powers of the squarefree ideal $I$ may also be described
in terms of a polyhedron.  Let $m=\prod x_{i}^{e_{i}}$ be a monomial and $P$ a monomial 
prime.  Then we have $m\in P$ if and only if $\displaystyle{\sum_{x_{i}\in
  P}e_{i}\geq 1}$ and $m\in P^{s}$ if and only if $\displaystyle{\sum_{x_{i}\in
  P}e_{i}\geq s}$.  This motivates the definition of the symbolic
polyhedron:

\begin{definition}\label{d:Symbolicpolyhedron}
  Let $I$ be a squarefree monomial ideal, and write $I=\cap_{P\in
    \Ass(I)}P$.  The \emph{symbolic polyhedron} associated to $I$ is
  the intersection of the half-spaces coming from the associated
  primes:
  \[
  \SP(I)=\bigcap_{P\in\Ass(I)}\{(e_{1},\dots, e_{n}):\sum_{x_{i}\in
    P}e_{i}\geq 1\}.
  \]
\end{definition}

\begin{remark}
  Like the Newton polyhedron, the symbolic polyhedron of $I$ is unbounded and 
  contained in the first orthant.  Unlike the Newton polyhedron, our
  definition works only for squarefree ideals.  For the definition of
  the symbolic polyhedron in the general monomial case, see the paper
  of Cooper, Embree, H\`{a}, and Hoefel \cite{CEHH17}.
\end{remark}

The following characterization of symbolic powers in terms of the
symbolic polyhedron is immediate.

\begin{proposition}\label{p:Symbolicpoly}
Suppose that $\SP$ is the symbolic polyhedron of $I$, and fix a
monomial $m$.  Then $m=\mathbf{x}^\mathbf{v}\in I^{(s)}$ if and only if
$\frac{\mathbf{v}}{s}\in \SP$.
\end{proposition}

Every convex polyhedron may be presented either as the convex hull of
some collection of points or as the intersection of some collection
of half-spaces.  Unfortunately, translating between the two
presentations requires solving a linear programming problem.  Since
the Newton polyhedron is presented as a convex hull and the symbolic
polyhedron is presented as an intersection, the difference between the
two is unavoidably somewhat opaque.  The difficulty of the general
linear programming problem likely accounts for the difficulty of the
questions involving these relationships.

\begin{example}\label{e:triangle}
Let $I=(ab,ac,bc)$, the standard example of a squarefree ideal whose
symbolic and regular powers are different.  We have $I=(a,b)\cap
(a,c)\cap (b,c)$.

The symbolic polyhedron of $I$ is defined as the intersection of the
three half-spaces $\SP=\{(a,b,c):a+b\geq 1, a+c\geq 1, b+c\geq 1\}$.
Observe that $\SP$ contains the point $\left(\frac{1}{2}, \frac{1}{2},
\frac{1}{2}\right)$.  This corresponds to the fact that $abc\in I^{(2)}$.

The Newton polyhedron of $I$ is defined as the convex hull of the
points corresponding to the monomials $ab$, $ac$, and $bc$ (together with infinitely many other points deeper in the first orthant).  It does
not contain the point $\left(\frac{1}{2}, \frac{1}{2}, \frac{1}{2}\right)$ since
the three generators cut out the affine hyperplane $a+b+c=2$.  In
fact, the Newton polyhedron is the intersection of the four
half-spaces $a+b\geq 1$, $a+c\geq 1$, $b+c\geq 1$, and $a+b+c\geq 2$.

We may usefully view the symbolic polyhedron as the intersection of
four half-spaces as well:  $a+b\geq 1$, $a+c\geq 1$, $b+c\geq 1$, and
the redundant $a+b+c\geq \frac{3}{2}$.  The distinction is that the
Newton polyhedron uses $a+b+c\geq 2$ because $2$ is the solution to the
integer linear program ``minimize $a+b+c$ subject to the other three
inequalities'' while the symbolic polyhedron uses $a+b+c\geq
\frac{3}{2}$ because $\frac{3}{2}$ is the solution to the
corresponding rational linear program.
\end{example}

We introduce some notation to deal with half-spaces and their defining
equations.

\begin{notation}\label{n:skewWaldschmidt}
  Let $v:\mathbb{R}^{n}\to \mathbb{R}$ be a (positive semidefinite)
  linear functional on the space of exponent vectors.  We refer to $v$
  as a \emph{skew valuation} on $S$.  For every skew valuation and
  every positive $c$, we define an associated hyperplane and
  half-space,
  \[
  P_{v,c}=\{\mathbf{v}: v(\mathbf{v})=c\} \qquad \text{and} \qquad
  H_{v,c}=\{\mathbf{v}:v(\mathbf{v})\geq c\}.
  \]
  The hyperplane $P_{v,c}$ is thus the boundary of the half-space
  $H_{v,c}$.  For a polyhedron $C$ in the first orthant, we say that
  $H_{v,c}$ is a \emph{supporting half-space} of $C$ if $H_{v,c}$
  contains $C$ and $P_{v,c}$ has nontrivial intersection with $C$.  We
  say that $H_{v,c}$ is a \emph{defining half-space} of $C$ if in
  addition $C\cap P_{v,c}$ has codimension one.

  Now fix a skew valuation $v$ and a monomial ideal $I$, and set
  $v(I)=\min\{v(\mathbf{v}):m=\mathbf{x}^{\mathbf{v}}\in 
  I\}$.  
  Since $v$ is positive semi-definite, $v(I)\ge 0$.  If $v(I)>0$, we say that
  $v$ is \textit{supported} on $I$.
  
  The \emph{skew Waldschmidt constant} $\widehat{v}(I)$ is the limit
  \[
  \widehat{v}(I)=\lim_{s\to\infty}\frac{v(I^{(s)})}{s},
  \]
  which exists because $\{v\left(I^{(s)}\right)\}$ is subadditive.
  If $v$ is supported on $I$, then $\widehat{v}(I)>0$ (this is true quite generally from Swanson's result
  that $I^{(sh)}\subset I^s$ for some constant $h$ --- see Lemma~\ref{lem:nonzeroskewwaldschmidt}).
  If $v$ is supported on $I$, the \emph{skew resurgence} of $I$ with respect to $v$ is
  \[
  v_{a}(I)=\frac{v(I)}{\widehat{v}(I)}.
  \]
\end{notation}

\begin{remark}
We use the term \textit{valuation} in Notation~\ref{n:skewWaldschmidt} because linear functionals $v:\ZZ^n\to\ZZ$ are naturally identified with \textit{monomial valuations} on $S$ (see~\cite[Definition~6.14]{IntegralClosure}).  We abuse notation somewhat by considering linear functionals on $\R^n$ instead of $\ZZ^n$, but this is harmless and helps to preserve geometric intuition.
\end{remark}

Note that when $v=\alpha$ is defined by $\alpha(\mathbf{v})=\deg \mathbf{x}^{\mathbf{v}}$, the skew Waldschmidt constant coincides with the standard Waldschmidt constant $\widehat{\alpha}$. 

Our main result will be that the asymptotic resurgence of $I$ is equal
to the maximum of its skew resurgences. We begin with a technical lemma whose proof is straightforward.

\begin{lemma}\label{l:hyperplanes}  Let $I$ be a squarefree monomial
  ideal with Newton and symbolic polyhedrons $\NP$ and $\SP$
  respectively, and write $I$ as an intersection of its associated
  primes, $I=\displaystyle\bigcap_{P\in\Ass(I)}P$.  For each $P\in \Ass(I)$, set
  $\displaystyle{v_{P}=\sum_{x_{i}\in P}e_{i}}$.  Then:

  \begin{enumerate}
  \item The defining half-spaces of $\SP$ are precisely the half-spaces
    $H_{v_{P},1}$ for $P\in \Ass(I)$.
  \item Every half-space of the form $H_{v_{P},1}$ is a defining
    half-space of $\NP$.

    Furthermore, for an arbitary skew valuation
    $v$ with integer coefficients, let $z_{v}$ be the solution to the
    integer linear programming problem ``minimize $v$, subject to the
    constraints $H_{v_{P},1}$ for all $P$.''  Then the half-space
    $H_{v,z_{v}}$ contains $\NP$, and every defining half-space of $\NP$
    arises in this way.

  \item For an arbitrary skew valuation
    $v$ with integer coefficients, let $q_{v}$ be the solution to the
    rational linear programming problem ``minimize $v$, subject to the
    constraints $H_{v_{P},1}$ for all $P$.''  Then the half-space
    $H_{v,q_{v}}$ is a supporting half-space of $\SP$.

  \item For an arbitrary skew valuation $v$ with integer coefficients,
    define $q_{v}$ as above.  Then $q_{v}$ is equal to the skew
    Waldschmidt constant, and we have
    \[
    \widehat{v}(I)=q_{v}=\min\left\{\frac{v(I^{(s)})}{s}\right\}.
    \]

  \item Fix a monomial $m=\mathbf{x}^{\mathbf{v}}$.  Then $m\in
    I^{(s)}$ if and only if $v(\mathbf{v})\geq s\widehat{v}(I)$ for
    all skew valuations $v$.  
  \end{enumerate}
\end{lemma}

\begin{theorem}\label{t:squarefreeresurgence}
Let $I$ be a squarefree ideal.  The asymptotic resurgence of $I$ is
equal to the maximum of its skew resurgences.  That is,
\[
\rho_{a}(I)=\sup\{v_{a}(I): v \text{ is a skew valuation supported on } I\}.
\]
\end{theorem}
\begin{proof}
  Recall from Remark~\ref{r:IntClosure} that we are using the characterization $\rho_a(I)=\sup \left\{\frac{s}{r}: I^{(s)} \not \subset \overline{I^r} \right\}$ which is proved in Section~\ref{s:AsymptoticResurgence}.
  
  To prove $\rho_{a}(I)$ is at most the maximum of the skew resurgences, it
  suffices to show that, whenever $\frac{s}{r}$ is greater than any skew
  resurgence, we have $I^{(s)}\subset \overline{I^{r}}$.  Fix such an
  $r$ and $s$ (meaning $s > rv_a(I)$ for all $v$), and let $m=\mathbf{x}^{\mathbf{v}}\in I^{(s)}$.  Then
  for every $v$ we have $v(\mathbf{v})\geq s\widehat{v}(I)
  > r v_{a}(I)\widehat{v}(I) =rv(I)$, so $m\in
  \overline{I^{r}}$.

To prove $\rho_{a}(I)$ is at least the maximum of the skew resurgences, suppose to the contrary that $\rho_{a}(I) < v_{a}$ for some $v$. Choose $r$ and $s$ such that \[\rho_a(I) < \frac{s}{r} < v_a = \frac{v(I)}{\widehat{v}(I)}.\] Since the inequality is strict, we have \[  \frac{s}{r} < \frac{v(I)}{\frac{v(I^{(st)})}{st}} \] for sufficiently large $t$. Rearranging, $v(I^{(st)}) < rtv(I)=v(\overline{I^{rt}})$. Consequently, $I^{(st)} \not \subseteq \overline{I^{rt}}$, so $\frac{st}{rt} \le \rho_a(I)$, a contradiction.
\end{proof}

\begin{corollary}\label{c:finitemax}
Let $I$ be a squarefree ideal.  The asymptotic resurgence of $I$ is
equal to the maximum of its skew resurgences, taken over the (finitely
many) skew valuations supported on $I$ that correspond to defining half-spaces of the Newton
polyhedron.  In particular, $\rho_{a}(I)$ is rational.
\end{corollary}

\begin{remark}
Theorem~\ref{t:squarefreeresurgence} and Corollary~\ref{c:finitemax} both follow immediately from Theorem~\ref{thm:generalizedWBound}; we need only recognize that the Rees valuations of a monomial ideal are precisely the skew valuations that correspond to defining half-spaces of the Newton polyhedron.  For a proof of this fact, see~\cite[Theorem~10.3.5]{IntegralClosure}.  The proofs given above emphasize the geometric intuition at hand for squarefree monomial ideals, which is not available in general.
\end{remark}

The rest of the section describes an inductive approach to the asymptotic resurgence of a squarefree monomial ideal.  While it is not clear that Theorem \ref{t:algorithm} would yield a more efficient algorithm for computing asymptotic resurgence, it is of theoretical and practical use (as we will see in Section~\ref{s:Examples}).

\begin{proposition}\label{p:parallelfacet}
Let $I$ be a squarefree monomial ideal with Newton polyhedron $\NP$. If $I$ is equigenerated in degree $d$, then $\NP$ has at most one facet not parallel to a coordinate axis, namely $P_{\alpha,d}\cap \NP$, where $\alpha$ is the valuation $\alpha(\mathbf{v}) = \deg \mathbf{x}^{\mathbf{v}}$.

More generally, suppose that the exponent vectors of all generators of $I$ are contained in some hyperplane $P_{v,c}$.  Then $\NP$ has at most one facet not parallel to a coordinate axis,  namely $P_{v,c} \cap \NP$.
\end{proposition}

\begin{proof}
We prove the first statement; the general case is similar. Suppose $\NP$ has some other facet $F$ not parallel to a coordinate axis, and choose $v$ and $c$ such that $F \subseteq P_{v,c}$. Let $m=\mathbf{x}^{\mathbf{v}} \in I$ be such that $\mathbf{v} \in F$. We claim that $m$ is a minimal generator for $I$. Indeed, suppose not. Then $m=x_i m'$ for some $m' \in I$. It follows that both $m'$ and $x_i^2 m'$ are contained in $I$, so (abusing notation) $
v(m') \ge c$ (so $v(x_i) \le 0$) and $v(x_i^2 m') \ge c$ (so $v(x_i) \ge 0$). Consequently, $v(x_i)=0$, a contradiction.

Therefore $F$ is the convex hull of minimal generators, each of which has degree $d$. Hence $F$ is contained in the two distinct hyperplanes $P_{v,c}$ and $P_{\alpha,d}$, contradicting the assumption that it is a facet of $\NP$. 
\end{proof}

\begin{lemma}\label{lem:projection}
The symbolic and Newton polyhedra commute with localization. More precisely, let $I \subseteq S=k[x_1,\dots,x_n]$ be a squarefree monomial ideal with symbolic and Newton polyhedra $\SP$ and $\NP$ respectively. Set $S'=k[x_1,\dots,x_{n-1}]$, and let $\pi: S \to S'$ be the map that evaluates $x_n$ to 1. Denote by $I'$, $\SP'$, and $\NP'$ the ideal $\pi(I)$ and its associated polyhedra. Then if $\pi_{\ast}: \mathbb{R}^n \to \mathbb{R}^{n-1}$ is the natural projection, we have:
\begin{enumerate}
\item $I'=(I:x_n^{\infty})$.
\item $\pi_{\ast}(\SP)=\SP'$.
\item $\pi_{\ast}(\NP)=\NP'$.
\end{enumerate}
\end{lemma}

\begin{proof}
Statement (1) is standard, and (3) follows immediately from (1). For (2), it suffices to show that $(I^{(s)}:x_n^{\infty}) = (I')^{(s)}$. Take a monomial $m \in S'$. Then:
\begin{align*}
m \in (I')^{(s)} 
& \Longleftrightarrow m \in P^s \text{ for all } P \in \Ass(I') \\
& \Longleftrightarrow m \in P^s \text{ for all } P \in \{Q \in \Ass(I) : x_n \not \in Q \}\\
& \Longleftrightarrow \text{for large } t, mx_n^{t} \in P^s \text{ for all } P \in \{Q \in \Ass(I) : x_n \not \in Q \} \tag{\dag}\\
& \Longleftrightarrow \text{for large } t, mx_n^{t} \in P^s \text{ for all } P \in \{Q \in \Ass(I) : x_n \not \in Q \} \tag{\ddag}\\  
& \qquad \qquad \quad \text{ and } mx_n^{t} \in P^s \text{ for all } P \in \{R \in \Ass(I) : x_n \in R \}\\
& \Longleftrightarrow \text{for large } t, mx_n^{t} \in P^s \text{ for all } P \in \Ass(I)\\
& \Longleftrightarrow m \in (I^{(s)}:x_n^{\infty}),
\end{align*}

where the equivalence between $(\dag)$ and $(\ddag)$ follows from the fact that $x_n^s\in R^s$. 
\end{proof}

\begin{notation}\label{n:contraction}
  Fix a monomial ideal $I$ and a collection of variables $U$.  Set $x_{U}=\displaystyle\prod_{x_{i}\in U} x_{i}$ and $I_{U}=(I:x_{U}^{\infty})$.  We call $I_{U}$ the \emph{saturation} of $I$ at $U$, or (because it can be attained by setting $x_{i}=1$ for all $x_{i}\in U$) the \emph{contraction} of $I$ at $U$.
\end{notation}

\begin{theorem}\label{t:algorithm}
  Suppose $I$ is a squarefree monomial ideal.  Then we may compute the asymptotic resurgence of $I$ as
  \[
  \rho_{a}(I)=\max\left\{\max_{U}\left\{\rho_{a}(I_{U})\right\}, \max_{v\in V}\left\{v_{a}(I)\right\}\right\},
  \]
  where $U$ runs over all subsets of variables and $V$ consists of the positive definite skew valuations that support defining half-spaces of the Newton polyhedron.  (That is, $V=\{v:v(x_{i})>0 \text{ for all } i \text{ and } P_{v,c}\cap \NP \text{ has dimension } n-1\}$.)
\end{theorem}
\begin{proof}
  Let $R$ be the value on the right-hand side.  Recall that $\rho_{a}(I)$ is the maximum, over all skew valuations, of $v_{a}(I)$.

  To show that $\rho_{a}(I)\leq R$, let $w$ be a skew valuation with $w_{a}(I)=\rho_{a}(I)$.  If $w$ is positive definite, we have $w\in V$ so $\rho_{a}(I)\leq R$.  If $w$ is not positive definite, set $U$ equal to the set of all $x_{i}$ with $w(x_{i})=0$.  Let $\pi$ be the projection map on exponent vectors as in Lemma \ref{lem:projection}.  Then we have $w(\mathbf{v})=w(\pi(\mathbf{v}))$ for all $\mathbf{v}$.  It follows that $w(I_{U})=w(I)$, $\widehat{w}(I_{U})=\widehat{w}(I)$, and $w_{a}(I_{U})=w_{a}(I)$.  Consequently, $\rho_{a}(I_{U})\geq w_{a}(I_{U})=\rho_{a}(I)$ and so $R\geq \rho_{a}(I)$. 

  To show that $\rho_{a}(I)\geq R$, we must show that $\rho_{a}(I_{U})\leq \rho_{a}(I)$ for all $U$.  Fix a $U$ and a skew valuation $w$ such that $w_{a}(I_{U})=\rho_{a}(I_{U})$.  Without loss of generality, we may assume that $w(x_{i})=0$ for all $x_{i}\in U$.  Consequently, $w(\mathbf{v})=w(\pi(\mathbf{v}))$ for all exponent vectors $v$, so $w_{a}(I)=w_{a}(I_{U})$.  We conclude that $\rho_{a}(I_{U})=w_{a}(I_{U})=w_{a}(I)\leq \rho_{a}(I)$ as desired.
\end{proof}

\begin{corollary}\label{c:lowdimensionalNewtonpolytope}
  Let $I$ be a squarefree monomial ideal.  Set $\rho_{c}(I)$ equal to the maximum of the asymptotic resurgences of its contractions, $\rho_{c}(I)=\max\{\rho_{a}(I_{U}): U\neq \varnothing\}$.  Denote by $L$ the affine span of the (exponent vectors of the) generators of $I$.  Then
  \begin{enumerate}
  \item Suppose $\dim L=n-1$.  Let $v$ be a nonzero valuation that is constant on $L$.  Then $\rho_{a}(I)=\max\{v_{a}(I),\rho_{c}(I)\}$.
  \item If $\dim L<n-1$ or, more generally, all facets of $NP(I)$ are parallel to a coordinate axis, then $\rho_{a}(I)=\rho_{c}(I)$.
  \end{enumerate}
\end{corollary}

\begin{remark}
It can certainly happen that many facets of the Newton polyhedron of $I$ are not parallel to any coordinate axis, in which case Corollary \ref{c:lowdimensionalNewtonpolytope} tells us nothing.  However, there are many cases of interest which satisfy one of the hypotheses of Corollary~\ref{c:lowdimensionalNewtonpolytope}, as we will see in the examples.
\end{remark}

\begin{proof}
We use Proposition~\ref{p:parallelfacet}.

  For statement (1), observe that $V$ consists of the single valuation $v$.

  For statement (2), observe that $V$ is empty.
\end{proof}

\begin{example}\label{ex:Jt}	
  Let $I=( x_0x_n,x_1x_n,\ldots,x_{n-1}x_n,x_0x_1\cdots x_{n-1})$.  To illustrate the utility of Corollary~\ref{c:lowdimensionalNewtonpolytope}, we prove that
  $\rho_a(I)=\frac{n^2}{n^2-n+1}$.  (In fact, it is straightforward to prove that $I$ is normal (that is, $\overline{I^{r}}=I^{r}$ for all $r$), so by Corollary~\ref{cor:rhoa=rho}, we have $\rho(I)=\frac{n^{2}}{n^{2}-n+1}$ as well.)  This was proven in~\cite[Theorem~7.14]{Waldschmidt16}.

Any contraction of $I$ will yield a monomial complete intersection, so the maximum of $\rho(I_U)$ taken over nonempty subsets of the variables is $1$.  Notice that the Newton polytope of $I$ is $(n-1)$-dimensional, lying in the plane with equation
\[
x_0+\cdots+x_{n-1}+(n-1)x_n=n.
\]
Let $v=x_0+\cdots+x_{n-1}+(n-1)x_n$.  By Corollary~\ref{c:lowdimensionalNewtonpolytope}, $\rho_a(I)=\rho_v(I)$.

Clearly $v(I)=n$.  We compute $\widehat{v}(I)$.  By Lemma~\ref{l:hyperplanes}, $\widehat{v}(I)=\min\{v(x):x\in\SP(I)\}$. $I$ has primary decomposition
\[
I=( x_0,x_n)\cap( x_1,x_n)\cap\cdots\cap( x_{n-1},x_n)\cap( x_0,\ldots,x_{n-1}).
\]

So we must minimize $\widehat{v}=x_0+\cdots+x_{n-1}+(n-1)x_n$ subject to the defining inequalities of $\SP(I)$, namely $x_{i}+x_{n}\geq 1$ for all $i\neq n$ and $x_{0}+\dots+x_{n-1}\geq 1$.  The skew Waldschmidt constant $\widehat{v}(I)$ is this minimum, namely $\frac{n^{2}-n+1}{n}$, attained at the point $(\frac{1}{n},\dots, \frac{1}{n},\frac{n-1}{n})$.  

Corollary  \ref{c:lowdimensionalNewtonpolytope} computes the asymptotic resurgence of $I$,
\[
\rho_a(I)=\frac{v(I)}{\widehat{v}(I)}=\frac{n^2}{n^2-n+1}.
\]
\end{example}

\begin{remark}\label{r:nonsquarefree}
  Although we have proven Theorem~\ref{t:algorithm} and Corollary~\ref{c:lowdimensionalNewtonpolytope} for squarefree monomial ideals, we remark that these results also hold for arbitrary monomial ideals without embedded primes. The following example illustrates this for a geometrically meaningful ideal.
\end{remark}

\begin{example}\label{ex:cremona}
  Let $I=( x^d,x^{d-1}y,y^{d-1}z)$.  This is the base ideal of a degree $d$ Cremona map from $\mathbb{P}^2$ to $\mathbb{P}^2$; symbolic powers of such ideals are studied in~\cite{CSR14}.  It satisfies $I^i=I^{(i)}$ for all $i<d$ but $I^{d}\neq I^{(d)}$ (see~\cite[Example~5.3]{A16} for details).  $I$ is Cohen-Macaulay with primary decomposition
\[
I=( x^d,x^{d-1}y,y^{d-1})\cap( x^{d-1},z).
\]
The Newton polytope of $I$ lies in the plane $x+y+z=d$ (with the obvious abuse of notation), and the contractions of $I$ are either $( x,y)$-primary or $( x,z)$-primary (hence $\rho_a(I_U)=1$ for every $U\subset \{x,y,z\}$).  Hence by Corollary~\ref{c:lowdimensionalNewtonpolytope}, $\rho_a(I)=\alpha_{a}(I)$.

The symbolic polyhedron $\SP(I)$ is given by the coordinate inequalities $x\geq 0$, $y\geq 0$, $z\geq 0$ and the more interesting inequalities 
\[
\begin{array}{rl}
(d-1)x+dy & \ge d(d-1)\\
x+(d-1)z & \ge d-1.\\
\end{array}
\]
Observe that these inequalities correspond in a natural way to the primary components, but the specifics are somewhat opaque.  For example, $x^{d-1}y$ is a generator of the $(x,y)$-primary component but is on the interior of the corresponding inequality.

By Lemma~\ref{l:hyperplanes}, $\widehat{\alpha}(I)$ is the minimum value of $x+y+z$ on $\SP$.  Some arithmetic yields 
$\walpha(I)=\frac{d^2-1}{d}$, so
\[
\rho_a(I)=\frac{\alpha(I)}{\walpha(I)}=\frac{d^2}{d^2-1}.
\]
Notice that $1<\rho_a(I)$, but we must take $t\ge d$ before we see the containment failure $I^{(t)}\not\subset \overline{I^t}$ guaranteed by Lemma~\ref{lem:rhoainequalities}.  Computational experiments suggest that $t=d$ usually suffices.
\end{example}

\section{Counterexamples and special cases}\label{s:Examples}
	 
	We begin this section with three examples of squarefree monomial ideals whose resurgence and asymptotic resurgence are not equal.  We then study the asymptotic resurgence of edge ideals (that is, squarefree monomial ideals generated by quadrics) and show that in this special case it depends only on the Waldschmidt constant.

	\begin{example}\label{e:Fano}
		Let $I=(abd,bce,cdf,aef,acg,deg,bfg)$ be the ideal whose generators correspond to the lines of the Fano plane.
		
		We compute that $\rho_{a}(I_{U})=1$ for all nontrivial $U\subseteq \{a,b,c,d,e,f,g\}$ (this also follows because $I$ is minimally non-Fulkersonian; see Remark~\ref{r:mfmc}).  Since $I$ is equigenerated, Corollary \ref{c:lowdimensionalNewtonpolytope} tells us $\rho_{a}(I)=\alpha_{a}(I)=\frac{\alpha(I)}{\widehat{\alpha}(I)}$.  We compute $\widehat{\alpha}(I)=\frac{7}{3}$ and $\alpha(I)=3$, so $\rho_{a}(I)=\frac{9}{7}$.  On the other hand, we have $abcdefg\in I^{(3)}\smallsetminus I^{2}$, so $\rho(I)\geq \frac{3}{2}$.
		
		Computing the resurgence is more delicate and requires results from Section~\ref{s:AsymptoticResurgence}.  We can verify computationally that the integral closure of the Rees algebra, $\overline{R[It]}=S\oplus \overline{I}\oplus \overline{I^2}\oplus \cdots$, is generated as a module over $R[It]$ by $\overline{I^2}$.  Hence $\overline{I^{r+1}}=I^{r-1}\overline{I^2}$ for $r\ge 1$.  Since $I$ is integrally closed, $\overline{I^2}\subset I$, hence $\overline{I^{r+1}}=I^{r-1}\overline{I^2}\subseteq I^{r}$ for $r\ge 2$.
		
		Using Proposition~\ref{prop:rhoupperbound}, we see that 
		\[
		\rho(I)\le\max\limits_{r\ge 1}\left\lbrace\frac{\lceil\rho_a(I)(r+1)\rceil-1}{r}\right\rbrace=\max\limits_{r\ge 1}\left\lbrace\frac{\lceil\frac{9}{7}(r+1)\rceil-1}{r}\right\rbrace,
		\]
		which is at most $\frac{3}{2}$ except when $r=3$.  However, we can check directly that $I^{(5)}\subset I^3$.  It follows that $\rho(I)=\frac{3}{2}$.  
	\end{example}
	
	\begin{example}\label{e:truncatedK5}
		Let $G$ be the hypergraph with vertices corresponding to the edges of the complete graph $K_{5}$ and hypergraph edges corresponding to the triangles in $K_{5}$.  The edge ideal of $G$ is $T=(abc,ade,bdf,cef,agh,bgi,chi,dgj,ehj,fij)$.
		
		Reasoning as in Example \ref{e:Fano}, we compute $\rho_{a}(T)=\alpha_{a}(T)=\frac{6}{5}$, but $\rho(T)\ge\frac{4}{3}$.  (In this case, the product of the variables $abcdefghij$ is contained in $T^{(4)}$ but not $T^{3}$.)
	\end{example}
	
	\begin{example}\label{e:truncatedK5dual} 
		Let $G$ be the hypergraph with vertices corresponding to the edges of the complete graph $K_{5}$ and hypergraph edges corresponding to the triangles in $K_{5}$, as in Example \ref{e:truncatedK5}, let $G^{\vee}$ be its Alexander dual.  Let \[T^{*}=(cefg, bdfh, afgh, adei, begi, cdhi, abcj, cdgj, behj, afij).\]  This is the ideal generated by the degree four generators of the edge ideal of $G^{\vee}$.  Reasoning once again as in Example \ref{e:Fano}, we compute $\rho_{a}(T^{*})=\frac{6}{5}$ and $\rho(T^{*})\ge\frac{3}{2}$.
	\end{example}
	
	\begin{remark}\label{r:mfmc}
		The three ideals of Examples \ref{e:Fano}, \ref{e:truncatedK5}, and \ref{e:truncatedK5dual} are known pathologies in the combinatorial optimization literature.  See \cite{C01} for details; we explain very briefly, abusing some notation, below.
		
		A squarefree monomial ideal $I$ (or more precisely its corresponding hypergraph) satisfies the \emph{max flow-min cut} (MFMC) property if $I^{(s)}=I^{s}$ for all $s$, and it satisfies the Fulkersonian property (or \textit{rounding up} property) if $I^{(s)}=\overline{I^{s}}$ for all $s$.  
		By Corollary~\ref{cor:rhoa1}, a squarefree monomial ideal is Fulkersonian if and only if its asymptotic resurgence is equal to $1$.
		
		The MFMC property is conjectured to be equivalent to another property, the \emph{packing property}, which is defined in terms of minors (that is, an ideal satisfies the packing property if it satisfies certain conditions after arbitrary subsets of the variables are set to zero or one); see \cite{C01} for more discussion of the packing property and the conjectured equivalence.  We say that an ideal is \emph{minimally non-packing} if it is not packed but every proper minor is.  Analogously, we say that an ideal is \emph{minimally non-Fulkersonian} if it is not Fulkersonian but every proper minor is.

		According to the remarks preceding \cite[Theorem 4.18]{C01}, every known hypergraph which is minimally non-Fulkersonian but not minimally non-packing contains one of the three hypergraphs from Examples \ref{e:Fano}, \ref{e:truncatedK5}, and \ref{e:truncatedK5dual} as a sub-hypergraph.
	\end{remark}
	
	In view of the strong relationship between asymptotic resurgence and combinatorial optimization described in Lemma \ref{l:hyperplanes}, we speculate that the two types of pathologies may be closely related.  
		
	\begin{question}\label{q:whynot}
	Suppose $I$ is minimally non-Fulkersonian but not minimally non-packing.  Is it the case that $\rho(I)\neq \rho_{a}(I)$? 
	\end{question}
	
	The converse to Question~\ref{q:whynot} is false; the edge ideal of the finite projective plane $\mathbb{P}^2(\mathbb{F}_3)$ also has asymptotic resurgence not equal to its resurgence.  In view of this, we ask the following question.
	
	\begin{question}\label{q:FiniteProjectivePlanes}
		Is the asymptotic resurgence of the edge ideal of a finite projective plane always different from its resurgence?
	\end{question}
	
	\begin{remark}
		Let $I_q$ be the edge ideal of the finite projective plane $\mathbb{P}^2(\mathbb{F}_q)$, where $q$ is a prime power.  One can show that the product of the variables is in $I_q^{(q+1)}$ but not in $I_q^2$, so $\rho(I_q)\ge\frac{q+1}{2}$.  It is also straightforward to show that
		$\walpha(I_q)=\frac{q^2+q+1}{q+1}$, so $\alpha_a(I_q)=\frac{(q+1)^2}{q^2+q+1}$.  However, the Waldschmidt constant does not always determine the asymptotic resurgence (for example, if $q=3$).  In any case, it seems reasonable to conjecture that $\rho(I_q)-\rho_a(I_q)\to\infty$ as $q\to\infty$.
	\end{remark}
	
	The above examples show that the asymptotic resurgence can be strictly smaller than the resurgence even for very well-behaved ideals.  A natural question is whether there is any large class of ideals (for instance, edge ideals of graphs) for which the asymptotic resurgence is equal to the resurgence.  The best result we can currently prove, even for edge ideals, is Corollary~\ref{cor:rhoa=rho} (which holds for arbitrary ideals); if an ideal is \textit{normal} then the asymptotic resurgence and resurgence are equal.
	
	The normality of edge ideals of graphs has been completely classified in~\cite{OH98,SVV98}.  We say a simple graph $G$ satisfies the \textit{odd cycle condition} if $G$ has no induced subgraph which consists of two disjoint cycles of odd length.  Then the edge ideal $I=I(G)$ is normal if and only if $G$ satisfies the odd cycle condition.  Hence, by Corollary~\ref{cor:rhoa=rho}, $\rho_a(I)=\rho(I)$ if $G$ satisfies the odd cycle condition.
	
	\begin{remark}
	Normality for more general squarefree ideals is quite difficult; see~\cite{HL15} for a discussion of this problem.
	\end{remark}
	
	We now show that the asymptotic resurgence of the edge ideal of a graph can be computed from the Waldschmidt constant.  We require some combinatorial tools, and a lemma relating containment of ideals to their symbolic powers.
        
	\begin{definition}\label{l:graphbasics}
		For a (hyper)graph $G$ and a collection of vertices $U$, recall that the \emph{link} of $U$ is $\link(U)=\{x_{i}\not\in U:x_{i}x_{j}\text{ is an edge for some $x_{j}\in U$}\}$, and the \emph{star} of $U$ is $\STAR(U)=U\cup \link(U)$.  Finally, a set of vertices $V$ is called a \emph{vertex cover} for $G$ if it intersects every edge.  The minimal vertex covers of $G$ are precisely the associated primes of the edge ideal $I(G)$.
	\end{definition}
	
	\begin{lemma}\label{l:obvo}
	  If $I$ and $J$ are squarefree monomial ideals and $I\subset J$, then $I^{(s)}\subset J^{(s)}$.  In particular, $\walpha(I)\ge \walpha(J)$.
        \end{lemma}

        Lemma \ref{l:obvo} seems obvious, but does not hold in general.  However, for squarefree monomial ideals it follows from careful consideration of the symbolic polyhedron.

        \begin{proof}[Proof of Lemma \ref{l:obvo}]
          Let $G$ and $H$ be the hypergraphs associated to $I$ and $J$ respectively. We first note that any vertex cover of $H$ is also a vertex cover of $G$. This follows from the fact that every edge of $G$ contains an edge of $H$.
          For an arbitrary vertex cover $\mathcal{C}$ of $H$, define the skew valuation $v_{\mathcal{C}}$ by $\displaystyle v_{\mathcal{C}}(\prod x_{i}^{e_{i}})=\sum_{x_{i}\in \mathcal{C}}e_{i}$.
          The defining half-spaces of the symbolic polyhedron of $J$ all have the form $H_{v_{\mathcal{C}},1}$ for some $\mathcal{C}$, and each such half-space contains the symbolic polyhedron of $I$.  Consequently, the symbolic polyhedron of $I$ is contained in the symbolic polyhedron of $J$.
        \end{proof}

	\begin{theorem}\label{t:graphs}
		Suppose $I=I(G)$ is the edge ideal of a graph $G$.  Then the asymptotic resurgence of $I$ is given by $\rho_{a}(I)=\frac{2}{\widehat{\alpha}(I)}$, where $\widehat{\alpha}(I)$ is the Waldschmidt constant.
	\end{theorem}
	\begin{proof} 
		We induct on the number of edges.  If there is only one edge, $I$ is principal and has $\rho_{a}(I)=1$ and Waldschmidt constant $\widehat{\alpha}(I)=2$ as desired.
		
		In general, observe that the exponent vectors of the generators of $I$ all satisfy $\alpha(\mathbf{v})=2$, where $\alpha$ is the degree valuation.  Consequently, by Corollary \ref{c:lowdimensionalNewtonpolytope} we have either $\rho_{a}(I) = \alpha_{a}(I)= \frac{\alpha(I)}{\widehat{\alpha}(I)}$ or $\rho_{a}(I)=\rho_{a}(I_{U})$ for some proper subset of variables $U$.  We must show that $\alpha_{a}(I)\geq \rho_{a}(I_{U})$ for all proper $U$.  
		
		Observe that $I_{U}=J+I(G')$, where $J=\link(U)$ and $G'$ is the induced subgraph on $V=\{x_{i}:x_{i}\not\in \STAR(U)\}$.  Since $J$ is generated by variables, we may conclude from \cite[Lemma~7.4]{Waldschmidt16} that $\rho_{a}(I_U)=\rho_{a}(I(G'))$, which is equal to $\alpha_{a}(I(G'))$ by induction.  Since $I(G')\subset I$, we have $\widehat{\alpha}(I(G'))\geq \widehat{\alpha}(I)$ by Lemma~\ref{l:obvo}, so
		\[
		\rho_{a}(I_{U}) = \rho_{a}(I(G')) =\frac{2}{\widehat{\alpha}(I(G'))} \leq \frac{2}{\widehat{\alpha}(I)},
		\]
		which was what we wanted.
	\end{proof}
	
	\begin{remark}
		It is possible to prove that the defining half-spaces for the Newton polyhedron of $I$ all arise from a skew valuation of the form $v_{U}(\prod x_{i}^{e_{i}}) = 2\sum_{x_{i}\in \link(U)}e_{i} + \sum_{x_{j}\not\in\STAR(U)}e_{j}$ for some subset $U$.  This allows us to prove Theorem \ref{t:graphs} using Corollary \ref{c:finitemax} instead of Corollary \ref{c:lowdimensionalNewtonpolytope}.
	\end{remark}
	
	In~\cite{Waldschmidt16} a nice relationship is shown between the \emph{fractional chromatic number} and the Waldschmidt constant.  We briefly recall here the definition of the fractional chromatic number of a graph $G$. If $\mathcal{I}$ is the set of independent sets of $G$ and $\mathbb{R}_{\geq 0}$ the set of nonnegative real numbers, a \emph{fractional coloring} is a function $f : \mathcal{I} \to \mathbb{R}_{\geq 0}$ satisfying, for every vertex $v$ of $G$, 
	\[
	\sum_{v \in A \in \mathcal{I}} f(A) \geq 1. 
	\]
	The fractional chromatic number $\chi_f(G)$ is then defined by
	\[
	\chi_f(G) = \inf \left\{ \sum_{A \in \mathcal{I}} f(A) : f \text{ is a fractional coloring of } G\right\}. 
	\]
	There is another (possibly more intuitive) description of the fractional chromatic number involving graph maps and the class of Kneser graphs $\text{KG}_{n,k}$. For details, see Chapter 17 in \cite{K08}. 
	
	The fractional chromatic number is bounded above by the classical chromatic number. For an example in which they are unequal, consider the $5$-cycle $C_5$. The function $f$ which takes the value of $\frac{1}{2}$ on an independent set of size $2$ and $0$ on an independent set of size $1$ is a fractional coloring. Thus, $\chi_f(C_5) \leq \frac{5}{2}$, whereas $\chi(C_5) = 3$. 
	
	If $I=I(G)$ is the edge ideal of a graph, then it is shown in~\cite[Theorem~4.6]{Waldschmidt16} that $\walpha(I)=\frac{\chi_f(G)}{\chi_f(G)-1}$ (this is shown more generally for any hypergraph but we will only need the statement for graphs).  Combining this result with Theorem~\ref{t:graphs}, we obtain:
	
\begin{corollary}\label{c:Fractional}
If $I=I(G)$ is the edge ideal of a graph, then $\rho_a(I)=\frac{2(\chi_f(G)-1)}{\chi_f(G)}.$
\end{corollary}
	
	\begin{remark}
		There are many graphs for which the fractional chromatic number is known (and hence by~\cite{Waldschmidt16}, the Waldschmidt constant of the corresponding edge ideal).  These include Kneser graphs, complete $k$-partite graphs, cycles, and wheels.  See~\cite[Section~6]{Waldschmidt16} for further discussion.  By Theorem~\ref{t:graphs} or Corollary~\ref{c:Fractional}, we automatically get the asymptotic resurgence of any such ideal.
	\end{remark}
		
	It is natural to ask whether Theorem~\ref{t:graphs} and Corollary~\ref{c:Fractional} can be extended to the resurgence. Van Tuyl asked the following question in Oaxaca in May 2017, which inspired the work in this paper.
	
	\begin{question}\label{q:resurgencefrac}
	If $I=I(G)$ is the edge ideal of a graph, is it true that $\rho(I)=\frac{2(\chi_f(G)-1)}{\chi_f(G)}$?
	\end{question}
	
	From the discussion prior to Theorem~\ref{t:graphs}, we can answer Question~\ref{q:resurgencefrac} in the affirmative if $I$ is normal (equivalently, $G$ satisfies the odd cycle condition).  If $I$ is not normal our methods do not appear strong enough to prove equality of the resurgence and asymptotic resurgence.  For instance, the following proposition shows that the resurgence does coincide with the  asymptotic resurgence for the edge ideal of the disjoint union of two odd cycles.  However, we need better containment results for the ideal of each individual cycle than those afforded by Lemma~\ref{lem:rhoainequalities}; for these we rely on the preprint~\cite{GHOS18}.
	
	\begin{proposition}\label{prop:twooddcycles}
	Suppose $G$ is a simple graph obtained as a disjoint union of two odd cycles $C_1$ and $C_2$, of lengths $n$ and $m$ respectively, where $n=2k+1<2\ell+l=m$. Then $\rho(I(G))=\rho_a(I(G))=\rho(I(C_1))=\frac{n+1}{n}$.
	\end{proposition}
	\begin{proof}
		By Corollary~\ref{c:lowdimensionalNewtonpolytope}, $\rho_a(I)=\alpha_a(I)=\frac{2}{\walpha(I)}$.  By~\cite[Corollary~4.7]{Waldschmidt16}, $\walpha(I)=\min\{\walpha(I_1),\walpha(I_2)\}$.  By~\cite[Theorem~6.7]{Waldschmidt16}, $\walpha(I_1)=\frac{n}{k+1}$ and $\walpha(I_2)=\frac{m}{\ell+1}$, hence $\rho_a(I)=\frac{2(k+1)}{n}=\frac{n+1}{n}$.
		
		Since $\rho_a(I)\le \rho(I)$, we need to prove that $\rho(I)\le \frac{n+1}{n}$.  It suffices to show that if $\frac{s}{r}>\frac{n+1}{n}$, then $I^{(s)}\subset I^r$.  Set $I_1=I(C_1)$ and $I_2=I(C_2)$.  We will use the following two results:
		\begin{enumerate}
			\item $(I_1+I_2)^{(s)}=\sum_{j=0}^s I_1^{(s-j)}I_2^{(j)}$ (\cite[Theorem~7.8]{Waldschmidt16})
			\item If $J=I(C)$ is the edge ideal of a cycle of length $n=2k+1$, $J^{(s)}\subset J^{r}$ if and only if $r<s-\lfloor\frac{s-(k+1)}{n+1}\rfloor$ (this can be deduced from~\cite[Theorem~3.4]{GHOS18})
		\end{enumerate}
		From (1) it suffices to show that $I_1^{(u)}I_2^{(v)}\subset I^r$ when $u+v=s$ and $\frac{s}{r}>\frac{n+1}{n}$.  If $r<u-\lfloor\frac{u-(k+1)}{n+1}\rfloor$, then by (2) $I_1^{(u)}\subset I_1^r$ hence $I_1^{(u)}I_2^{(v)}\subset I^r$.  So we assume $r\ge u-\lfloor\frac{u-(k+1)}{n+1}\rfloor$.  Put $h=u-\lfloor\frac{u-(k+1)}{n+1}\rfloor-1$.  Then $r=h+c$, where $c>0$.  By (2), $I_1^{(u)}\subset I^h$.  It suffices to show that $I_2^{(v)}\subset I^c$, since then $I_1^{(u)}I_2^{(v)}\subset I^hI^c=I^r$.  A straightforward but tedious computation starting with the inequality $\frac{s}{r}=\frac{u+v}{h+c}>\frac{n+1}{n}$ yields that $c<v-\lfloor\frac{v-(\ell+1)}{m+1}\rfloor$ (remember $n\le m=2\ell+1$); by (2) this proves that $I_2^{(v)}\subset I^c$.
	\end{proof}
	
	We proved in Section \ref{s:squarefree} that the resurgence and asymptotic resurgence of square-free ideals are bounded below in terms of the degrees of generators.  We close the current section with a (generally very coarse) upper bound, also in terms of the degrees of generators. 
	
	\begin{theorem}\label{t:upperbound}
		Let $I$ be a squarefree ideal.  Then $\rho(I)\leq\omega(I)$, where $\omega(I)$ is the maximum degree of a minimal generator of $I$.
	\end{theorem}
	\begin{proof}
		It suffices to show that $I^{(s)}\subseteq I^{r}$ whenever $r\leq \frac{s}{\omega}$.  To this end, fix any such $r$ and $s$, and suppose $m\in I^{(s)}$.  Write $m=m'\beta$ with $\beta\not\in I$ and $m'\in I^{k}\smallsetminus I^{k+1}$.  We will show $k\geq r$.
		
		Let $\mathcal{H}$ be the hypergraph whose edges are the generators of $I$.
		Since $\beta\not\in I$, the support of $\beta$ cannot contain any edge of $\mathcal{H}$, so $V=\{x_{1},\dots, x_{n}\}\smallsetminus \supp(\beta)$ must be a vertex cover of $\mathcal{H}$.  Writing $m'=\mathbf{x}^{\mathbf{v}}=\prod x_{i}^{e_{i}}$, we conclude $\sum_{x_{i}\in V}e_{i}\geq s$.  This yields the inequality
		\[
		s\leq \sum_{x_{i}\in V}e_{i}\leq \deg m' \leq \omega k,
		\]
		where the final inequality follows since $m'\in I^k$. So $k\geq \frac{s}{\omega}\geq r$ as desired.
	\end{proof}
	
	We learned in personal communication that H\`a and Trung have recently independently proven a result stronger than Theorem~\ref{t:upperbound}, recovering the theorem as a corollary.
	
	\begin{remark}
	Suppose $I=I(G)$ is the edge ideal of a graph.  Using Corollary~\ref{c:Fractional} and Theorem~\ref{t:upperbound}, we have the following bounds on $\rho(I)$:
	\[
	\frac{2(\chi_f(G)-1)}{\chi_f(G)}\le \rho(I)\le 2,
	\]
	with equality on the left if $I$ is normal.
	\end{remark}

\section{Asymptotic resurgence and integral closure}\label{s:AsymptoticResurgence}
In this section we prove our main results, showing that the asymptotic resurgence of an ideal in the polynomial ring may be computed using integral closures.  From this we will derive the fact that asymptotic resurgence is the maximum of finitely many Waldschmidt-like constants; as in the monomial case we call these \textit{skew Waldschmidt constants}.  

We will assume throughout that $I$ is a homogeneous ideal in the polynomial ring $S=\kk[x_1,\ldots,x_n]$.  We begin with a lemma which provides upper bounds on $\rho_a(I)$.

\begin{lemma}\label{lem:asymptoticupperbound}
Suppose $\{s_i\}$ and $\{r_i\}$ are sequences of positive integers such that $\lim s_i=\lim r_i=\infty$, $I^{(s_i)}\subseteq I^{r_i}$ for all $i$, and
\[
\lim \frac{s_i}{r_i}=h
\]
for some $h\in\R$.  Then $\rho_a(I)\le h$.
\end{lemma}
\begin{proof}
We proceed by contradiction.  Suppose $\rho_a(I)>h$. Then there exists a rational number $\frac{s}{r}$, $h<\frac{s}{r}<\rho_{a}$, such that $I^{(st)}\not\subset I^{rt}$ for all $t\gg 0$.

Now, for all $i$ large enough, $\frac{s_i}{r_i}<\frac{s}{r}$, so $sr_i-rs_i>0$.  Also, we claim 
$\lim sr_i-rs_i=\infty$:  
if not we would have $\frac{s}{r}-h=\lim \frac{s}{r}-\frac{s_i}{r_i}=\lim \frac{sr_i-rs_i}{rr_i}=0$ (since the denominator goes to infinity).  It would follow that $\frac{s}{r}=h$, contradicting the construction of $\frac{s}{r}>h$.  

Let $t_{0}$ be such that for all $t\ge t_0$, we have $I^{(st)}\not\subset I^{rt}$.  Choose $i$ with $r_i\ge rt_0$ and $sr_i-rs_i>rs$.  Now let $t$ be maximal such that $r_i\ge rt$, and observe $t\ge t_0$.  By the choice of $t$ we have $r_{i}<r(t+1)$, so $sr_{i}<srt+sr$.  By the choice of $i$ we have $sr_{i}>rs_{i}+rs$.  Combining these inequalities yields $rs_{i}<rst$, so $s_{i}<st$.
We conclude
\[
I^{(st)}\subseteq I^{(s_i)}\subseteq I^{r_i} \subseteq I^{rt}
\]
(the first containment by the inequality derived above, the second by the assumptions on $\{s_{i}\}$ and $\{r_{i}\}$, and the third by the choice of $t$).  In particular, $I^{(st)}\subseteq I^{rt}$, so (by construction of $t_{0}$) we must have $t<t_{0}$, a contradiction.  
\end{proof}

If $I\subset S$ is an ideal we denote by $\overline{I}$ the \textit{integral closure} of $I$, which is the set of all elements $r\in S$ which satisfy an \textit{equation of integral dependence} over $I$.  That is
\[
r^n+a_1r^{n-1}+a_2r^{n-2}+\cdots +a_n=0,
\]
where $n$ is a positive integer and $a_i\in I^{i}$ for $i=1,\ldots,n$.  Our reference for this topic is the book of Swanson and Huneke~\cite{IntegralClosure}.  We introduce two statistics related to resurgence and integral closure as follows:
\[
\overline{\rho}(I):=\sup\left\lbrace\frac{s}{r}: I^{(s)}\not\subset \overline{I^r}\right\rbrace \quad\mbox{and}\quad \overline{\rho}_a(I):=\sup\left\lbrace\frac{s}{r}: I^{(st)}\not\subset \overline{I^{rt}} \mbox{ for all } t\gg 0\right\rbrace.
\]
A primary result of this section is that both of these statistics are equal to the asymptotic resurgence. Clearly $\overline{\rho}_a(I) \le \rho_a(I)$ and $\overline{\rho}(I)\le \rho(I)$, with equality if all powers of $I$ are integrally closed (equivalently, if the Rees algebra of $I$ is integrally closed).

\begin{proposition}\label{prop:rhoa=rhoabar}
Let $I$ be an ideal.  Then $\rho_a(I)=\overline{\rho}_a(I)$.
\end{proposition}
\begin{proof}
It is clear that $\overline{\rho}_a(I)\le \rho_a(I)$, so it suffices to show $\rho_a(I)\le \overline{\rho}_a(I)$.  To this end, 
choose a rational $h=\frac{s}{r}>\overline{\rho}_{a}(I)$.  We will show that $\rho_{a}(I)\leq h$.  

By construction, $I^{(st)}\subset \overline{I^{rt}}$ for infinitely many $t$.
By \cite[Proposition 5.3.4]{IntegralClosure} the integral closure of the Rees algebra of $I$ is finitely generated over the Rees algebra of $I$, so there is some integer $k$ such that $\overline{I^n}=I^{n-k}\overline{I^k}\subset I^{n-k}$ for all $n\ge k$ (see~\cite[Proposition~5.3.4]{IntegralClosure}).

Now let $\{t_{i}\}$ be an increasing sequence satisfying $I^{(st_{i})}\subset \overline{I^{rt_{i}}}$ for all $t_{i}$ (which must exist by the construction of $h=\frac{s}{r}$).  We have $I^{(st_{i})}\subset \overline{I^{rt_{i}}}\subset I^{rt_{i}-k}$ for all $i$, so the sequences $\{s_{i}=st_{i}\}$ and $\{r_{i}=rt_{i}-k\}$ satisfy the hypotheses of Lemma~\ref{lem:asymptoticupperbound}.  We conclude $\rho_{a}(I)\leq \frac{s}{r}=h$.
\end{proof}

\begin{remark}\label{rem:BrianconSkoda}
  Using the Brian\c{c}on-Skoda Theorem (see~\cite[Theorem~13.3.3]{IntegralClosure} and the following remarks), we can choose the integer $k$ independently of the ideal $I$ in the proof of Proposition~\ref{prop:rhoa=rhoabar}.  More precisely, we can always take $k=n-1$, regardless of the ideal $I$.
\end{remark}

Recall that a \textit{discrete valuation} on a field $\mathbf{K}$ is a homomorphism $v:\mathbf{K}^*=\mathbf{K}\smallsetminus\{0\}\to\ZZ$ from the multiplicative group $\mathbf{K}^*$ to the additive group $\ZZ$ satisfying that for all $x,y\in\mathbf{K}$, $v(x+y)\ge\min\{v(x),v(y)\}$.  If $S$ is the polynomial ring, we will take $\mathbf{K}$ to be the fraction field of $S$.  In this case, a valuation on $\mathbf{K}$ is determined uniquely by its values on $S$, so we will abuse notation by referring to these as valuations on $S$ rather than $\mathbf{K}$.  Given a valuation $v$, set $\mathbf{K}_v=\{x\in\mathbf{K}:v(x)\ge 0\}$.  Then $V=\mathbf{K}_v$ is a discrete valuation ring (DVR) with field of fractions $\mathbf{K}$ and we denote its unique maximal ideal by $\m_V$.  If $V\subset \mathbf{K}$ is a DVR with maximal ideal $\m_{V}$, we may define a valuation $v_{V}$ by the rule $v_{V}(x)=\max(k:x\in \m_{V}^{k})$.  Every valuation with valuation ring $V$ is then a scalar multiple of $v_{V}$, so there is a one-to-one correspondence between valuations on $\mathbf{K}$ (up to this equivalence) and DVRs whose field of fractions is $K$.  For more details, see \cite[Chapter 6]{IntegralClosure}.

Given a valuation $v$ and an ideal $I$, we write $v(I)$ for the minimum value that $v$ takes on $I$; i.e., $v(I)=\min\{v(f):f\in I\}$.  Valuations are  relevant for our analysis because of the valuative criterion for integral closure (see \cite[Theorem 6.8.3]{IntegralClosure}):

\begin{theorem}[Valuative Criterion for Integral Closure]\label{t:VCfIC}
  Fix an ideal $I$, and $x\in S$.  Then $x\in \overline{I}$ if and only if $x\in IV$ for every DVR $V$ containing $S$ with field of fractions $\mathbf{K}$.    Equivalently, $x\in \overline{I}$ if and only if $v(x)\geq v(I)$ for every discrete valuation $v$.  Furthermore, if $J$ is another ideal, then $J\subset \overline{I}$ if and only if $v(J)<v(I)$ for all $v$.  
  \end{theorem}

We now discuss the sequence $\{v(I^{(n)})\}$.  Since $I^{(m)}I^{(n)}\subset I^{(m+n)}$, and $v$ is a valuation, it follows that $v$ is \textit{subadditive}.  (That is, $v(I^{(m+n)})\le v(I^{(m)})+v(I^{(n)})$ for all $m,n$.)  Thus $\displaystyle{\lim_{n\to\infty} \frac{v(I^{(n)})}{n}=\inf\left\{\frac{v(I^{(n)})}{n}\right\}}$. (This is sometimes called \textit{Fekete's lemma} and holds for any subadditive sequence.)  Another important consequence of subadditivity is that $\frac{v(I^{(n)})}{n}\le \frac{v(I^{(m)})}{m}$ whenever $m$ divides $n$.  We record these facts in the following lemma (see~\cite[Lemma~2.3.1]{BH10} for a proof).

\begin{lemma}\label{lem:generalizedWaldschmidt}
	Let $v:S\to\ZZ$ be a discrete valuation and $I\subset S$ an ideal.  Then the limit
	\[
	\widehat{v}(I):=\lim_{s\to\infty} \dfrac{v(I^{(s)})}{s}
	\]
	exists, and $\widehat{v}(I)=\inf\limits_s\{\frac{v(I^{(s)})}{s}\}$.  Moreover, $\frac{v(I^{(n)})}{n}\le \frac{v(I^{(m)})}{m}$ if $m$ divides $n$.
\end{lemma}

\begin{remark}
The map $\alpha:S\to \ZZ$ defined on homogeneous polynomials by $\alpha(f)=\deg(f)$ extends to a valuation on $S$.  In this case, $\walpha(I)$ is the Waldschmidt constant, and Lemma~\ref{lem:generalizedWaldschmidt} is the first part of ~\cite[Lemma~2.3.1]{BH10}. 
\end{remark}

\begin{definition}
Given a valuation $v$ and an ideal $I$, we say $v$ \emph{is supported on $I$} if $v(I)\ge 1$.
\end{definition}

\begin{lemma}\label{lem:nonzeroskewwaldschmidt}
A valuation $v$ is supported on $I$ if and only if $\widehat{v}(I)>0$.
\end{lemma}
\begin{proof}
  Suppose $\widehat{v}(I)>0$.  Then $v(I)\geq \widehat{v}(I)$ is an integer and in particular must be at least $1$.  

   Conversely, suppose $v(I)\ge 1$.  By~\cite{S00}, there is some $h$ so that $I^{(sh)}\subset I^s$ for all $s$.  It follows that
\[
\frac{v(I^{(sh)})}{sh}\ge \frac{v(I^s)}{sh} =\frac{sv(I)}{sh}\ge \frac{1}{h}>0,
\]
hence $\widehat{v}(I)>0$.
\end{proof}

\begin{remark}\label{r:ReesValuations}
  Every non-zero ideal in $S$ has a (unique) set of DVRs $V_1,\ldots,V_r$ (called \textit{Rees valuation rings}) so that
\begin{enumerate}
	\item $V_i\subset\mathbf{K}$, where $\mathbf{K}$ is the fraction field of $S$,
	\item for all $n\in\N$, $\overline{I^n}=\cap_{i=1}^r I^nV_i$, and
	\item the set $V_1,\ldots,V_r$ satisfying (2) is minimal possible.
\end{enumerate}
See~\cite[Chapter~10]{IntegralClosure} for details of the construction.  The corresponding valuations $v_1,\ldots,v_r$ are called \textit{Rees valuations} (these are unique up to equivalence for valuations; see~\cite[Definition~6.1.8]{IntegralClosure}).  Thus to check that $x\in \overline{I^n}$ using the valuative criterion for integral closure, it suffices to check that $v(x)\ge nv(I)$ for the finitely many Rees valuations of $I$.
\end{remark}

\begin{theorem}\label{thm:generalizedWBound}
Let $I$ be an ideal and let $v_1,\ldots,v_r$ be the set of Rees valuations for $I$.  Then
\[
\rho_a(I)=\max\limits_i\left\lbrace\frac{v_i(I)}{\widehat{v}_i(I)}\right\rbrace=\sup\limits_v\left\lbrace\frac{v(I)}{\widehat{v}(I)}\right\rbrace,
\]
where the maximum and supremum are taken over discrete valuations which are supported on $I$.
\end{theorem}
\begin{proof}
Write $M$ for $\displaystyle{\max_i\left\{\frac{v_i(I)}{\widehat{v}_i(I)}\right\}}$.  We first show that $\rho_a(I)\le M$.  Suppose to the contrary that $\rho_a(I)>M$.  By Proposition~\ref{prop:rhoa=rhoabar}, there exist $r$ and $s$ such that $M<\frac{s}{r}<\rho_{a}(I)$ and $I^{(s)}\not\subset \overline{I^{r}}$.  By the valuative criterion for integral closure and the properties of Rees valuations, there exists a Rees valuation $v_{i}$ such that $v_{i}(I^{(s)})<v_{i}(I^{r})=rv_{i}(I)$.  Now,
\begin{align*}
  M &\geq \frac{v_{i}(I)}{\widehat{v_{i}}(I)} &\text{(by assumption)}\\
  &\geq \frac{v_{i}(I)}{\frac{v_{i}(I^{(s)})}{s}} &\text{(by Lemma~\ref{lem:generalizedWaldschmidt})}\\
  &> \frac{v_{i}(I)}{\frac{rv_{i}(I)}{s}} &\text{(by the discussion above)}\\
  &=\frac{s}{r},
\end{align*}
a contradiction.

To show $M\leq \rho_{a}(I)$, suppose $v$ is any valuation supported on $I$.  We will show $\frac{v(I)}{\widehat{v}(I)}\leq \rho_{a}(I)$.  Suppose $r$ and $s$ are such that $\frac{s}{r}<\frac{v(I)}{\widehat{v}(I)}$.  Then there exists $t_{0}$ such that $\frac{s}{r}<\frac{v(I)}{\frac{v(I^{(st)})}{st}}$ for all $t\geq t_{0}$.  Hence $v(I^{(st)})< rtv(I)=v(I^{rt})$ for all $t\ge t_0$.  By the valuative criterion for integral closure, $I^{(st)}\not\subset \overline{I^{rt}}$ for $t\ge t_0$.  Hence $\frac{s}{r} \le \overline{\rho}_a(I)$, so by Proposition~\ref{prop:rhoa=rhoabar} $\frac{s}{r}\le \rho_a(I)$.  Since $\frac{s}{r}$ could be arbitrarily close to $\frac{v(I)}{\widehat{v}(I)}$, we conclude that $\frac{v(I)}{\widehat{v}(I)}\leq \rho_{a}(I)$.

We have shown $M\le\sup\left\{\frac{v(I)}{\widehat{v}(I)}\right\}\le \rho_a(I)\le M$.  The desired equalities are immediate.
\end{proof}

\begin{remark}
Theorem~\ref{thm:generalizedWBound} generalizes the well-known bound $\frac{\alpha(I)}{\walpha(I)}\le \rho(I)$.
\end{remark}

\begin{lemma}\label{lem:rhoainequalities}
	If $I$ is an ideal, then
	\begin{enumerate}
		\item if $I^{(s)}\not\subseteq \overline{I^r}$ then $\frac{s}{r}<\rho_a(I)$
		\item if $\frac{s}{r}<\rho_a(I)$ then $I^{(st)}\not\subseteq \overline{I^{rt}}$ for all $t\gg 0$.
	\end{enumerate}
\end{lemma}
\begin{remark}
Given an ideal $I$ and a fraction $\frac{s}{r}<\rho_a(I)$, we may have to take $t$ to be quite large before $I^{(st)}\not\subseteq \overline{I^{rt}}$.  See Example~\ref{ex:cremona}.
\end{remark}
\begin{proof}
	If $I^{(s)}\not\subseteq \overline{I^r}$ then by the valuative criterion for integral closure there is a valuation $v:S\to\ZZ$ so that $v(I^{(s)})<v(I^r)$.  So we have
	\[
	\frac{s}{r}<\frac{s}{r}\frac{v(I^{r})}{v(I^{(s)})} = \frac{s}{r}\frac{rv(I)}{v(I^{(s)})}=\frac{v(I)}{\frac{v(I^{(s)})}{s}}\le \frac{v(I)}{\widehat{v}(I)}\le\rho_a(I)
	\]
	by Lemma~\ref{lem:generalizedWaldschmidt} and Theorem~\ref{thm:generalizedWBound}.  This proves (1).
	
	For (2), if $\frac{s}{r}<\rho_a(I)$ then by Theorem~\ref{thm:generalizedWBound} there is a valuation $v:S\to\ZZ$ so that $\frac{s}{r}<\frac{v(I)}{\widehat{v}(I)}$.  By Lemma~\ref{lem:generalizedWaldschmidt},
	\[
	\frac{s}{r}<\frac{v(I)}{\frac{v(I^{(st)})}{st}}
	\]
	for all $t\gg 0$.  Rearranging, $v(I^{(st)})<rtv(I)=v(I^{rt})$ for all $t\gg 0$, so the valuative criterion for integral closure tells us that $I^{(st)}\not\subseteq \overline{I^{rt}}$ for all $t\gg 0$.
\end{proof}

\begin{corollary}\label{cor:rhoabar=rhobar}\label{cor:rhoa=rho}
For any ideal $I$, we have $\rho_{a}(I)=\overline{\rho}_{a}(I)=\overline{\rho}(I)$.  In particular, if $I$ is normal (that is, if all powers of $I$ are integrally closed), then $\rho_{a}(I)=\rho(I)$.
\end{corollary}
\begin{proof}
The inequality $\overline{\rho}_a(I)\le\overline{\rho}(I)$ is automatic from the definitions. 
The rest follows from part (1) of Lemma~\ref{lem:rhoainequalities} and  Proposition~\ref{prop:rhoa=rhoabar}.
\end{proof}

\begin{remark}\label{rem:smooth}
  Corollary~\ref{cor:rhoa=rho} gives a partial answer to the question raised at the end of~\cite{GHV13}: in what cases do we have $\rho_a(I)=\rho(I)$?  The normalilty hypothesis here is very strong and far from sharp
  -- it happens that $\rho_a(I)=\rho(I)$ for many ideals which are not normal.
\end{remark}

\begin{corollary}\label{cor:rhoa1}
Let $I$ be an ideal.  Then $\rho_a(I)\ge 1$, with equality if and only if $I^{(s)}\subseteq\overline{I^s}$ for every $s\ge 1$.
\end{corollary}
\begin{proof}
For the inequality $1\le \rho_a(I)$, see~\cite[Theorem~1.1]{GHV13}.  Now suppose there is some $s\ge 1$ so that $I^{(s)}\not\subset\overline{I^s}$.  Then by Lemma~\ref{lem:rhoainequalities}, $1<\rho_a(I)$.  Hence if $\rho_a(I)=1$, we must have $I^{(s)}\subseteq \overline{I^s}$ for every $s\ge 1$.
\end{proof}

\begin{corollary}\label{cor:rho=rhoa=1}
Suppose that the symbolic powers of $I$ are integrally closed (for example, if $I$ is radical).  Then $\rho(I)=1$ if and only if $\rho_a(I)=1$ and $\overline{I^{r+1}}\subset I^r$ for all $r\ge 1$.
\end{corollary}
\begin{proof}
Suppose first that $\rho(I)=1$.  It follows from $1\le \rho_a(I)\le \rho(I)$ that $\rho_a(I)=1$.  Since the symbolic powers of $I$ are integrally closed, Corollary~\ref{cor:rhoa1} yields $I^{(r)}\subseteq \overline{I^{r}}\subseteq \overline{I^{(r)}}=I^{(r)}$, so $I^{(r)}=\overline{I^r}$ for all $r\ge 1$.  Suppose that $\overline{I^{r+1}}=I^{(r+1)}\not\subset I^r$ for some $r\ge 1$.  Then $\rho(I)\ge \frac{r+1}{r}>1$, a contradiction.  Hence we must have $\overline{I^{r+1}}\subset I^r$ for $r\ge 1$. 

Now suppose that $\rho_a(I)=1$ and $\overline{I^{r+1}}\subset I^r$ for all $r\ge 1$.  As above, we have $I^{(s)}=\overline{I^s}$ for all $s\ge 1$. If $\rho(I)>1$, there would exist positive integers $s>r$ so that $I^{(s)}=\overline{I^s}\not\subset I^r$.  Since $\overline{I^s}\subseteq\overline{I^{r+1}}$, it would follow that $\overline{I^{r+1}}\not\subset I^r$, a contradiction.  So $\rho(I)\le 1$.  Since we always have $1\le \rho(I)$, this shows $\rho(I)=1$.
\end{proof}

We now discuss some upper bounds on resurgence in terms of asymptotic resurgence.

\begin{definition}
For a fixed positive integer $r$, write $K_r=K_r(I):=\min\{s:\overline{I^s}\subset I^r\}$.  By the Brian\c{c}on-Skoda theorem (see~\cite[Theorem~13.3.3]{IntegralClosure} and following remarks), $K_r\le r+(n-1)$, where $n$ is the number of variables of $S$. Hence
\[
1\le \frac{K_r}{r}\le 1+\frac{(n-1)}{r},
\]
and $\max_r\left\{\frac{K_r}{r}\right\}$ exists.  Write $K(I)$ for this maximum value.  The inequalities above show $K(I)\le n$.
\end{definition}

\begin{proposition}\label{prop:rhoupperbound}
For any ideal $I$,
\[
\rho(I)\le \max_r\left\lbrace\frac{\lceil \rho_a(I)K_r(I)\rceil-1}{r}\right\rbrace\le \rho_a(I)K(I)\le \rho_a(I)n.
\]
\end{proposition}
\begin{proof}
Suppose $I^{(s)}\not\subset I^r$.  Then $I^{(s)}\not\subset \overline{I^{K_r}}$, so by Lemma~\ref{lem:rhoainequalities}, $\frac{s}{K_r}<\rho_a(I)$, hence $s\le\lceil \rho_a(I) K_r\rceil-1$.  It follows that
$\frac{s}{r}\le \frac{\lceil \rho_a(I) K_r\rceil-1}{r}$, hence
\[
\rho(I)=\sup\left\lbrace\frac{s}{r}: I^{(s)}\not\subset I^r\right\rbrace\le \sup \left\lbrace\frac{\lceil \rho_a(I) K_r\rceil-1}{r}\right\rbrace.
\]
The result now follows from the inequalities
\[
\frac{\lceil \rho_{a}(I)K_{r}\rceil - 1}{r} \leq \frac{\rho_{a}(I)K_{r}}{r}\leq \rho_{a}(I)K(I).
\]
\end{proof}

\begin{remark}
The results of~\cite{ELS01} and~\cite{HH02} imply that $\rho(I)\le n-1$, rendering the final inequality in~\ref{prop:rhoupperbound} useless; however, either of the two expressions prior to the final inequality can represent non-trivial improvements, depending on the ideal.
\end{remark}


\newcommand{\etalchar}[1]{$^{#1}$}

\end{document}